\documentclass[final,1p,times]{elsarticle}

\usepackage{amssymb}
\usepackage{amsthm}
\usepackage{amsmath,amssymb,amsopn,amsfonts,mathrsfs,amsbsy,amscd}
\usepackage{longtable}

\newcommand{\ass}{\mathrm{ass}}

\newcommand{\prs}{\langle\;,\;\rangle}

\newcommand{\too}{\longrightarrow}
\newcommand{\om}{\omega}
\newcommand{\esp}{\quad\mbox{and}\quad}

\def\br{[\;,\;]}

\newcommand{\R}{\mathbb{R}}
\newcommand{\G}{{\mathfrak{g}}}

\newcommand{\h}{{\mathfrak{h}}}

\newcommand{\tr}{{\mathrm{tr}}}
\newcommand{\ric}{{\mathrm{ric}}}
\newcommand{\Li}{\mathrm{L}}
\newcommand{\Lii}{\overline{\mathrm{L}}}
\newcommand{\Ri}{\mathrm{R}}

\newcommand{\D}{{\cal D}}

\newcommand{\na}{\nabla}

\newcommand{\al}{\alpha}
\newcommand{\be}{\beta}
\newcommand{\ga}{\gamma}
\newcommand{\Ga}{\Gamma}
\newcommand{\e}{\epsilon}

\newcommand{\la}{\lambda}

\newtheorem{theo}{Theorem}[section]
\newtheorem{pr}{Proposition}[section]
\newtheorem{Le}{Lemma}[section]
\newtheorem{co}{Corollary}[section]

\newtheorem{exem}{Example}
\newtheorem{remark}{Remark}

\begin{document}

\begin{frontmatter}
	
	
	
	
	\title{  Real Left-Symmetric Algebras with Positive Definite Koszul Form and K\"ahler-Einstein Structures}
	
	\author[label1, label2]{Mohamed Boucetta, Hasna Essoufi }
	\address[label1]{Universit\'e Cadi-Ayyad\\
		Facult\'e des sciences et techniques\\
		BP 549 Marrakech Maroc\\e-mail: m.boucetta@uca.ma
	}
	
	\address[label2]{Universit\'e Cadi-Ayyad\\
		Facult\'e des sciences et techniques\\
		BP 549 Marrakech Maroc\\e-mail: essoufi.hasna@gmail.com
	}
	
	
	
	
	\begin{abstract} 
		
		Let \((\G, \bullet)\) be a real left symmetric algebra, and \((\G^-, \br)\) the corresponding Lie algebra. We denote by \(\Li\) the left multiplication operator associated with the product \(\bullet\). The symmetric bilinear form \(\mathrm{B}(X, Y) = \tr(\Li_{X \bullet Y})\), referred to as the Koszul form of \((\G, \bullet)\), is introduced. We provide a complete characterization, along with a broad class of examples, of real left symmetric algebras that possess a positive definite Koszul form. In particular, we show that for a left symmetric algebra with positive definite Koszul form being commutative or associative or Novikov implies that this algebra is isomorphic to $\R^n$ endowed with its canonical product.
		Beyond their  algebraic interest, we show that any real left symmetric algebra \((\G, \bullet)\) with a positive definite Koszul form induces a K\"ahler-Einstein structure with negative scalar curvature on the tangent bundle \(TG\) of any connected Lie group \(G\) associated to \((\G^-, \br)\). Furthermore, the characterization of left symmetric algebras with a positive definite Koszul form leads to a new class of non-associative algebras, which are of independent interest and generalize Hessian Lie algebras.
		
	\end{abstract}
	
	
\end{frontmatter}







\section{Introduction}\label{section1}

A Hessian manifold (for details, see \cite{shima, shima1}) is a triple $(M,g,\na)$ where $g$ is a Riemannian metric and $\na$ is a flat torsionless connection such that $g$ satisfies the Codazzi equation
\begin{equation}\label{codazzi} \na_X(g)(Y,Z)=\na_Y(g)(X,Z) \end{equation}for any vector fields $X,Y,Z$. Denote by $\D$ the Levi-Civita connection of $(M,g)$. The Koszul 1-form $\al$ and the second Koszul form $\be$ of $(M,g,\na)$ are given by
\begin{equation}\label{koszul} \al(X)=\tr(\ga_X)\esp \be(X,Y)=\na_X(\al)(Y),\quad X,Y\in\Ga(TM), \end{equation}where $\ga_XY=\D_XY-\na_XY$. The 1-form $\al$ is closed which implies that $\be$ is symmetric.
One of the key properties of Hessian manifolds is that their tangent bundle \(TM\) naturally admits a Kähler structure \((\hat{g}, J)\). Moreover, this structure is \(\mu\)-Einstein\footnote{A Riemannian manifold $(M,g)$ is called $\mu$-Einstein if its Ricci tensor $\ric$ satisfies $\ric=\mu g$.}if and only if \(\beta = -\mu g\).

A Hessian Lie group is a Lie group $G$ endowed with a left invariant Hessian structure $(g,\na)$. The couple $(g,\na)$ induces on the Lie algebra $(\G,\br)$ of $G$ identified to the vector space of left invariant vector fields, a scalar product $\prs$ and product $\bullet$ defined by
\[ \langle X,Y\rangle=g(X,Y)\esp X\bullet Y=\na_XY,\quad X,Y\in\G. \]
Since both $g$ and $\na$ are left invariant, the connection $\na$ is flat and torsion-free if and only if
\begin{equation}\label{left}
	X\bullet Y-Y\bullet X=[X,Y]\esp \ass(X,Y,Z)=\ass(Y,X,Z),\quad X,Y,Z\in\G,
\end{equation}where $\ass(X,Y,Z)=(X\bullet Y)\bullet Z-X\bullet(Y\bullet Z)$ is the associator. Additionally, the Codazzi equation \eqref{codazzi} is equivalent to:
\begin{equation}\label{hessian}
	\langle X\bullet Y-Y\bullet X,Z\rangle=\langle Y\bullet Z,X\rangle- \langle X\bullet Z,Y\rangle,\quad X,Y,Z\in\G.
\end{equation}Recall that an algebra $(\G,\bullet)$ is called a left symmetric algebra if its associator satisfies the second condition in \eqref{left}. It is well known that, in this case, $\bullet$ is Lie admissible, i.e, the bracket $[X,Y]=X\bullet Y-Y\bullet X$ is a Lie bracket. We introduce the Koszul form of $(\G,\bullet)$ as the bilinear symmetric form $\mathrm{B}$ given by
\begin{equation}\label{abel}
	\mathrm{B}(X,Y)=\tr(\Li_{X\bullet Y})
\end{equation}where $\Li$ is the left multiplication operator of $\bullet$. The left symmetry of the associator implies that $\mathrm{B}$ satisfies
\begin{equation}\label{6}
	\mathrm{B}(X\bullet Y-Y\bullet X,Z)=\mathrm{B}(Y\bullet Z,X)-\mathrm{B}(X\bullet Z,Y),\quad X,Y,Z\in\G.
\end{equation}

A Hessian  algebra is a left symmetric algebra $(\G,\bullet)$ endowed with a scalar product $\prs$ satisfying \eqref{hessian}. Left symmetric algebras and Hessian algebras play a significant role in geometry, physics, and algebra (see   \cite{bai, burde, koszul,Vinberg}).

Let \((G, g, \nabla)\) be a Hessian Lie group, and \((\G, \bullet, \langle \cdot, \cdot \rangle)\) its associated Hessian  algebra. The key observation here is that the second Koszul form of \((G, g, \nabla)\) is independent of the metric \(g\) and, when restricted to \(\G\), it coincides with the Koszul form \(B\) of \((\G, \bullet)\). Since we have established that the Kähler structure on \(TG\) is \(\mu\)-Einstein if and only if the second Koszul form \(\beta\) satisfies \(\beta = -\mu g\), we can now state our first main result:

\begin{theo}\label{main}
	\begin{enumerate}
		\item Let $(G,g,\na)$ be a Hessian Lie group such that the Koszul form of its associated Hessian algebra vanishes.Then the K\"ahler structure of $TG$ is Ricci-flat.

		\item 
		Let \((\G, \bullet)\) be a real left symmetric algebra with a positive definite Koszul form \(B\), and let \(G\) be a connected Lie group whose Lie algebra corresponds to the underlying Lie algebra of \((\G, \bullet)\). Then $(\G,\bullet,B)$ is a Hessian algebra, $G$ is solvable and, for any \(\alpha > 0\), \((\alpha B, \bullet)\) defines a left-invariant Hessian structure \((g, \nabla)\) on \(G\), and the associated Kähler structure on \(TG\) is \(-\frac{1}{\alpha}\)-Einstein.
	\end{enumerate}
\end{theo}
It is important to mention that there is Lie group structure on $TG$ such that the associated K\"ahler structure is left invariant (see \cite{bou}). As a consequence, the Koszul form $B$ can never be negative definite, otherwise $TG$ will have an Einstein metric with positive scalar curvature which is not possible since $TG$ is not compact.

Theorem \ref{main} underscores the significance of studying left symmetric algebras with a positive definite Koszul form for short $\mathrm{LSPK}$, beyond their inherent interest as a subclass of Hessian algebras. Consequently, the second part of this paper is dedicated to a complete examination of this class of non-associative algebras. We first show by using the famous Artin-Wedderburn theorem  that for a $\mathrm{LSPK}$ being associative or commutative or Novikov implies that the algebra is isomorphic to $\R^n$ with its canonical associative commutative product (see Theorem \ref{ass}). The general case is treated in
 Theorem \ref{theo} where a complete description of this class of algebras is given. This theorem shows that these algebras can be constructed from  a class of algebras that generalize Hessian algebras. To our knowledge, this broader class has not been previously considered. Let us introduce this class.

For $k\in\R$, a $k$-Hessian  algebra is an algebra $(\G,\bullet)$ endowed with a scalar product satisfying \eqref{hessian} and, for any $X,Y,Z\in\G$,
\begin{equation}\label{k}
	\ass(X,Y,Z)-\ass(Y,X,Z)=k\left(\langle X,Z\rangle Y-\langle Y,Z\rangle X\right).
\end{equation}
Clearly a $0$-Hessian algebra is a Hessian algebra. Note that the relation \eqref{k} implies that the bracket $[X,Y]=X\bullet Y-Y\bullet X$ defines a Lie bracket. Moreover, if $G$ is a connected Lie group whose Lie algebra is $(\G,\br)$ then $(\bullet,\prs)$ induces on $G$ a left invariant Riemannian metric $g$ and a left invariant torsion free connection $\na$ such that $g$ satisfies the Codazzi equation \eqref{codazzi} and, for any $X,Y\in\Ga(TG)$,
\begin{equation}\label{khessian}
	R^\na(X,Y):=\na_{[X,Y]}-\na_X\na_Y+\na_Y\na_X=kX\wedge Y,
\end{equation}where $(X\wedge Y)(Z)=\langle X,Z\rangle Y-\langle Y,Z\rangle X$. This leads to an important generalization of Hessian manifolds, referred to as $k$-Hessian manifolds.  A $k$-Hessian manifold is a Riemannian manifold $(M,g)$ endowed with a torsion-free connection $\na$ satisfying \eqref{codazzi} and \eqref{khessian}. Note that a Riemannian manifold with constant sectional curvature is a $k$-Hessian manifold if   $\na$ is the Levi-Civita connection.

In conclusion, the study of $\mathrm{LSPK}$ provides a valuable framework for constructing Einstein-Kähler manifolds. Additionally, this exploration has led to the discovery of two new structures: \(k\)-Hessian algebras and \(k\)-Hessian manifolds. These structures are of independent interest and merit further investigation on their own. We devote Section \ref{section4} to a preliminary study of $k$-Hessian algebras where we prove an important result (see Theorem \ref{milnor}) and we give many examples.

The paper is organized as follows. In Section \ref{section2}, to ensure the paper is self-contained, we calculate the Ricci curvature of the Kähler structure on the tangent bundle of a Hessian manifold using a method distinct from that in \cite{shima}, and provide a proof of Theorem \ref{main}. Section \ref{section3} is dedicated to the algebraic study of $\mathrm{LSPK}$, culminating in Theorem \ref{theo} and its corollary. In Section \ref{section4}, we introduce $k$-Hessian algebras, present several classes of examples, and characterize a subclass of these algebras (see Theorem \ref{milnor}). Section \ref{section6} addresses the classification of $\mathrm{LSPK}$ in dimensions 2 and 3, along with two examples in dimensions 4 and 5. Finally, Section \ref{section7} is an appendix that provides detailed computations required for the proof of Proposition \ref{mainpr}.

\section{Ricci curvature of the K\"ahler structure on the tangent bundle of a Hessian manifold and a proof of Theorem \ref{main}}\label{section2}

In this section, we describe the K\"ahler structure on the tangent bundle of a Hessian manifold and we compute its Ricci curvature.

Let $(M,g,\na)$ be a Hessian manifold of dimension $n$.  We denote by $\D$ the Levi-Civita connection of $g$, $K(X,Y)=\D_{[X,Y]}-\D_X\D_Y+\D_Y\D_X$ its curvature, $\ric$ its Ricci curvature and  $\ga$ the  difference tensor  given by
\begin{equation*}\label{ga1} \ga_XY=\D_XY-\na_XY,\quad X,Y\in\Ga(TM). \end{equation*} Since both $\na$ and $D$ are torsionless, $\ga$ is symmetric and it is easy to check that, for any $X,Y,Z\in\Ga(TM)$,
\begin{equation*}\label{nam}
	\na_{X}(g)(Y,Z)=g(\ga_XY+\ga^*_XY,Z).
\end{equation*}and hence the Codazzi equation \eqref{codazzi} is equivalent to $\ga=\ga^*$, i.e,
\[ g( \ga_XY,Z)=g( Y,\ga_XZ),\quad X,Y,Z\in\Ga(TM). \]
Let $\al$ and $\be$ the Koszul forms given by \eqref{koszul}. The 1-form $\al$ is closed and hence $\be$ symmetric. Denote by $H$ the vector field given by $g(H,X)=\al(X)$.

\begin{pr}\label{curvature} Let $(M,g,\na)$ be a Hessian manifold. Then, for any $X,Y,Z\in\Ga(TM)$,
	\begin{equation*}\begin{cases}\label{curv} 
			\D_X(\ga)(Y,Z)=\D_X(\ga)(Z,Y)=\D_Y(\ga)(X,Z),\\
			K(X,Y)=[\ga_X,\ga_Y],
			\ric(X,Y)=\tr(\ga_X\circ\ga_Y)-\tr(\ga_{\ga_XY}). \end{cases}\end{equation*}
	Moreover, for any orthonormal local frame $(E_1,\ldots,E_n)$,
	\begin{equation*}\label{H} H=\sum_{i=1}^n\ga_{E_i}E_i\esp \D_XH=\sum_{i=1}^n\D_X(\ga)(E_i,E_i). \end{equation*}
	
\end{pr}

\begin{proof} Since $\ga_XY=\ga_YX$ we have obviously  $\D_X(\ga)(Y,Z)=\D_X(\ga)(Z,Y)$. On the other hand, since $\ga=\ga^*$, we have, for any $T\in\Ga(TM)$,
	\begin{equation}\label{gamma}g( \D_X(\ga)(Y,Z),T)=g( \D_X(\ga)(Y,T),Z).
	\end{equation}
	Now, we have
	\begin{align*}
		K(X,Y)Z&=\D_{[X,Y]} Z-\D_X\D_YZ+\D_Y\D_XZ\\
		&=\na_{[X,Y]} Z+\ga_{[X,Y]} Z-\D_X\na_YZ-\D_X\ga_YZ+\D_Y\na_XZ
		+\D_Y\ga_XZ\\
		&=\na_{[X,Y]} Z+\ga_{\D_XY}Z-\ga_{\D_YX}Z
		-\na_X\na_YZ-\ga_X\na_YZ-\D_X\ga_YZ+\na_Y\na_XZ+\ga_Y\na_XZ+\D_Y\ga_XZ\\
		&=\ga_{\D_XY}Z-\ga_{\D_YX}Z-\ga_X\D_YZ+\ga_X\ga_YZ-\D_X\ga_YZ+\ga_Y\D_XZ-+\ga_Y\ga_XZ+\D_Y\ga_XZ\\
		&=\D_Y(\ga)(X,Z)-\D_X(\ga)(Y,Z)+[\ga_X,\ga_Y]Z.
	\end{align*} 
	Now $K(X,Y)$ and $[\ga_X,\ga_Y]$ are skew-symmetric with respect to $g$ and, by virtue of \eqref{gamma}, the tensor field $Z\mapsto \D_Y(\ga)(X,Z)-\D_X(\ga)(Y,Z)$ is symmetric so we must have
	\[ \D_Y(\ga)(X,Z)=\D_X(\ga)(Y,Z)\esp K(X,Y)=[\ga_X,\ga_Y]. \]The expression of the Ricci curvature follows immediately from the expression of $K$. 
	
	On the other hand, Fix a point $p\in M$. It is known that there exists   a local orthonormal frame $(E_1,\ldots,E_n)$ in a neighborhood of $p$ such that $(\D E_j)(p)=0$ for $j=1,\ldots,n$.  We have, for any $X\in\Ga(TM)$
	\begin{align*} &g(H,X)=\tr(\ga_X)=\sum_{i=1}^ng(\ga_XE_i,E_i)=
		\sum_{i=1}^ng(\ga_{E_i}E_i,X), \\
		&\sum_{i=1}^n\D_X(\ga)(E_i,E_i)
		=\D_X(H)-2\ga_{\D_XE_i}E_i.
	\end{align*} By evaluating at $p$ we get that the desired result.
\end{proof}
Let us describe now the K\"ahler structure on $TM$ associated to $(M,g,\na)$. Denote by $\pi:TM\too M$ the canonical projection. It is well-known that the connection $\na$ gives rise to a splitting 
\[ TTM=\ker T\pi\oplus \mathcal{H}, \]where
\[ \mathcal{H}_u=\{X^h(u),X\in T_{\pi(u)}M \}\esp X^h(u)=\frac{d}{dt}_{|t=0}\tau^t(u) \]where $\tau^t:T_{\pi(u)}M\too T_{\exp(tu)}M$ is the parallel transport associated to $\na$ along the $\na$-geodesic $t\mapsto\exp(tu)$.
For any $X\in\Ga(TM)$, we denote by $X^h$ its horizontal lift and by $X^v$ its vertical lift.   For any $X,Y\in\Ga(TM)$,
\begin{equation}\label{br} [X^h,Y^h]=[X,Y]^h,\;[X^h,Y^v]=(\na_XY)^v \esp[X^v,Y^v]=0. \end{equation} 
We define on $TM$ a Riemannian metric $\hat{g}$ and a complex structure $J$ by putting
\[ \begin{cases}
	\hat{g}(X^h,Y^h)=g(X,Y)\circ\pi,\; \hat{g}(X^v,Y^v)=g(X,Y)\circ\pi\esp \hat{g}(X^h,Y^v)=0,\\
	JX^h=X^v,\; JX^v=-X^h,\quad X,Y\in\Ga(TM).\end{cases}\] 
Then $(TM,\hat{g},J)$ is a K\"ahler manifold (see \cite{shima}).

By using \eqref{br} and the Koszul formula of the Levi-Civita connection, one can check easily that the Levi-Civita connection $\na^{LC}$ of $\hat{g}$  is given by
\[ \na^{LC}_{X^h}Y^h=(D_XY)^h,\; \na^{LC}_{X^v}Y^v=-(\ga_XY)^h,\;\na^{LC}_{X^v}Y^h=
(\ga_XY)^v \esp \na^{LC}_{X^h}Y^v=(D_XY)^v,\quad X,Y\in\Ga(TM).\] 
Let us compute the curvature of $(TM,\hat{g})$. 
\begin{pr} We have, for any $X,Y,Z\in\Ga(TM)$,
	\[ \begin{cases} R(X^h,Y^h)Z^h=(K(X,Y)Z)^h,\;
		R(X^h,Y^h)Z^v=(K(X,Y)Z)^v,\;
		R(X^v,Y^v)Z^h=([\ga_X,\ga_Y]Z)^h,\;\\
		R(X^v,Y^v)Z^v=([\ga_X,\ga_Y]Z)^v,\;
		R(X^h,Y^v)Z^h=-(\D_X(\ga)(Z,Y))^v
		-(\ga_Z\ga_XY)^v,\;
		R(X^h,Y^v)Z^v=
		(\D_X(\ga)(Y,Z))^h
		+\left( 
		\ga_{Z}\ga_XY \right)^h.
	\end{cases} \]
	
\end{pr}

\begin{proof} We have
	\begin{align*}
		R(X^h,Y^h)Z^h&=(K(X,Y)Z)^h,\\
		R(X^h,Y^h)Z^v&=\na^{LC}_{[X,Y]^h}Z^v
		-\na^{LC}_{X^h}\na^{LC}_{Y^h}Z^v
		+\na^{LC}_{Y^h}\na^{LC}_{X^h}Z^v\\
		&=(\D_{[X,Y]}Z)^v-\na^{LC}_{X^h}(\D_{Y}Z)^v+\na^{LC}_{Y^h}(\D_{X}Z)^v\\
		&=(K(X,Y)Z)^v,\\
		R(X^v,Y^v)Z^h&=-\na^{LC}_{X^v}\na^{LC}_{Y^v}Z^h+\na^{LC}_{Y^v}\na^{LC}_{X^v}Z^h\\
		&=-\na^{LC}_{X^v}(\ga_YZ)^v
		+\na^{LC}_{Y^v}(\ga_XZ)^v\\
		&=(\ga_X\ga_YZ)^h-(\ga_Y\ga_XZ)^h,\\
		&=([\ga_X,\ga_Y]Z)^h\\
		R(X^v,Y^v)Z^v&=
		-\na^{LC}_{X^v}\na^{LC}_{Y^v}Z^v+\na^{LC}_{Y^v}\na^{LC}_{X^v}Z^v\\
		&=([\ga_X,\ga_Y]Z)^v,\\
		R(X^h,Y^v)Z^h&=\na^{LC}_{(\na_XY)^v}Z^h
		-\na^{LC}_{X^h}(\ga_ZY)^v
		+\na^{LC}_{Y^v}(\D_XZ)^h\\
		&=(\ga_Z\na_XY)^v-(\D_X\ga_ZY)^V+(\ga_Y\D_XZ)^v\\
		&=(\ga_Z\D_XY)^v-(\D_X\ga_ZY)^v+(\ga_Y\D_XZ)^v-\ga_Z\ga_XY\\
		&=-(\D_X(\ga)(Y,Z))^v-(\ga_Z\ga_XY)^v\end{align*}
	\begin{align*}
		R(X^h,Y^v)Z^v&=
		\na^{LC}_{(\na_XY)^v}Z^v+\na^{LC}_{X^h}(
		\ga_YZ)^h
		+\na^{LC}_{Y^v}(\D_XZ)^v\\
		&=-(\ga_Z\D_XY)^h+(\ga_Z\ga_XY)^h+(\D_X\ga_YZ)^h-(\ga_Y\D_XZ)^h\\
		&=(\D_X(\ga)(Y,Z))^h+(\ga_Z\ga_XY)^h.\quad\quad\quad\quad\quad\quad\quad\quad\quad\quad\quad\qedhere
\end{align*}\end{proof}
The Ricci curvature of $(TM,\hat{g})$ is related to the second Koszul form $\be$.
\begin{pr}\label{einstein}We have, for any $X,Y\in\Ga(TM)$,
	\[ \ric^{\hat{g}}(X^h,Y^h)=
	\ric^{\hat{g}}(X^v,Y^v)=-\be(X,Y)\circ\pi\esp 
	\ric^{\hat{g}}(X^h,Y^v)=0.
	\]In particular, $(TM,\hat{g})$ is $\mu$-Einstein if and only if $\be=-\mu g$.
\end{pr}
\begin{proof} Note first that since $(TM,\hat{g},J)$ is a K\"ahler manifold then its Ricci curvature satisfies $\ric^{\hat{g}}(JU,JV)=\ric^{\hat{g}}(U,V)$ for any vector fields $U,V$ on $TM$. This implies that 
	$\ric^{\hat{g}}(X^v,Y^v)=\ric^{\hat{g}}(X^h,Y^h)$. 	Let $(E_1,\ldots,E_n)$ be a local orthonormal frame of $M$. We have
	\begin{align*}
		\ric^{\hat{g}}(X^h,Y^h)&=\sum_{i=1}^n\left(
		g(R(X^h,E_i^h))Y^h,E_i^h)+
		g(R(X^h,E_i^v))Y^h,E_i^v)\right)\\
		&=\ric(X,Y)\circ\pi-\sum_{i=1}^n\left(\langle\D_X(\ga)(Y,E_i),E_i\rangle\circ\pi+
		\langle \ga_Y\ga_X E_i,E_i\rangle\circ\pi
		\right)\\
		&=	\tr(\ga_X\circ \ga_Y)\circ\pi-\tr(\ga_{\ga_XY})\circ\pi
		-\sum_{i=1}^n\langle\D_X(\ga)(E_i,E_i),Y\rangle\circ\pi-\tr(\ga_X\circ\ga_Y)\circ\pi\\
		&\stackrel{\eqref{H}}=-\langle \D_XH,Y\rangle\circ\pi-\tr(\ga_{\ga_XY})\circ\pi\\
		&=-X.\langle H,Y\rangle +\langle H,\D_XY\rangle-\langle \ga_XY,H\rangle\\
		&=-X.\al(Y)+\al(\na_XY)=-\na_X(\al)(Y),\\
		\ric^{\hat{g}}(X^h,Y^v)&=\sum_{i=1}^n\left(
		g(R(X^h,E_i^h))Y^v,E_i^h)+
		g(R(X^h,E_i^v))Y^v,E_i^v)\right)=0.\quad\quad\quad\quad\quad\quad\quad\quad\qedhere
	\end{align*}
\end{proof}

Let \((G, g, \nabla)\) be a Hessian Lie group and  \((\G, \prs, \bullet)\)  its associated Hessian algebra. The Levi-Civita connection of \(g\) induces a product \(\star\) on \(\G\), referred to as the Levi-Civita product. We define the endomorphisms \(\Li_X\) and \(\Lii_X\) by \(\Li_X Y = X \bullet Y\) and \(\Lii_X Y = X \star Y\), respectively. The Levi-Civita product \(\star\) is characterized by the properties that \(X \star Y - Y \star X = [X, Y]\), and \(\Lii_X\) is skew-symmetric with respect to \(\prs\).

Interestingly, although the Koszul 1-form and the second Koszul form generally depend on both \(\nabla\) and \(g\), in the case of Hessian Lie groups, they depend only on \(\nabla\). This is demonstrated by the following key result.

\begin{pr}\label{beta} Let $(G,\na,\prs)$ be a Hessian Lie group. Then its Koszul 1-form and second Koszul form are given by
	\[\al(X)=-\tr(\Li_X)\esp  \be(X,Y)=\tr(\Li_{X\bullet Y}),\quad X,Y\in\G. \]
\end{pr}

\begin{proof} For any $X\in\G$, $\ga_X=\Lii_X-\Li_X$ and since $\Lii_X$ is skew-symmetric, $\al(X)=\tr(\ga_X)=-\tr(\Li_X)$. Moreover, for any $X,Y\in\G$,
	\[ \be(X,Y)=\na_{X}(\al)(Y)=-\al(\na_XY)=-
	\al(X\bullet Y)=\tr(\Li_{X\bullet Y}) \]which completes the proof.
\end{proof}

\subsection{Proof of Theorem \ref{main}}
\begin{proof}
	Note first that according to \cite[Corollary 4]{shima1} that underlying Lie algebra of a Hessian algebra is solvable and
Theorem \ref{main} follows immediately from   Propositions \ref{einstein}-\ref{beta}.\end{proof}

\section{Left symmetric algebras with  positive definite Koszul form.}\label{section3}

Theorem \ref{main} emphasizes the importance of $\mathrm{LSPK}$. Accordingly, we dedicate this section to a complete study of this class of algebras culminating in Theorem \ref{theo}.

We start this study by addressing two important cases, namely, the associative case and the Novikov case.

Remark first that if a $\mathrm{LSPK}$ is commutative then it is associative. 
Recall that a left symmetric algebra $(\G,\bullet)$ is called Novikov if, for any $X,Y,Z\in\G$,
\begin{equation}\label{novikov} (X\bullet Y)\bullet Z=(X\bullet Z)\bullet Y. \end{equation}
Novikov algebras constitute an important subclass of left symmetric algebras and have been studied by many authors (see \cite{burde0, burde1}). As a consequence of \eqref{novikov}, the Koszul form of a Novikov algebra satisfies
\[ \mathrm{B}(X\bullet Y,Z)=\mathrm{B}(X\bullet Z,Y),\quad X,Y,Z\in\G. \]This relation combined with \eqref{6} implies that
\[ \mathrm{B}([X,Y],Z)=0,\quad X,Y,Z\in\G. \]
So if a $\mathrm{LSPK}$ is Novikov then its commutative and hence associative.
\begin{exem} Consider $(\R^n,\bullet)$ endowed with its canonical associative commutative product:
	\[ X\bullet Y=(X_1Y_1,\ldots,X_nY_n). \]
	It is easy to check that the Koszul form of 
	$(\R^n,\bullet)$ is the canonical Euclidean product of $\R^n$. Moreover, $(\R^n,\bullet)$ is a Novikov.
	
\end{exem}
It turns out that it is the only example of $\mathrm{LSPK}$ which is Novikov, commutative or associative.
\begin{theo}\label{ass} Let $(\G,\bullet)$ be a $\mathrm{LSPK}$ algebra. Then the following assertions are equivalent.
	\begin{enumerate}\item[$(i)$] $(\G,\bullet)$ is a Novikov algebra.
		\item[$(ii)$] $(\G,\bullet)$ is a commutative algebra.
		\item[$(iii)$] $(\G,\bullet)$ is an associative algebra.
	\end{enumerate}	Moreover, in this case $(\G,\bullet)$ is isomorphic to $\R^n$ endowed with its canonical associative commutative product.
\end{theo}
\begin{proof} To prove the theorem, we will show that $(iii)$ implies that $(\G,\bullet)$ is isomorphic to $\R^n$ endowed with its canonical associative commutative product. Suppose that $(\G,\bullet)$ is an associative algebra with positive definite Koszul form. In this case the relation $\ass(X,Y,Z)=0$ implies that the Koszul form $\mathrm{B}$ satisfies
	\[ \mathrm{B}(X\bullet Y,Z)=\mathrm{B}(X,Y\bullet Z). \]
	Since $\mathrm{B}$ is nondegenerate, $\mathrm{B}$ becomes a trace form on $(\G,\bullet)$. This implies that if $I$ is an ideal of $(\G,\bullet)$ then its orthogonal $I^\perp$ is also an ideal and $\G$ is semi-simple and splits $\G=\G_1\oplus\ldots\oplus \G_r$ where each $\G_i$ is a simple   associative algebra with positive definite Koszul form. The classification of finite-dimensional simple associative algebras over \( \mathbb{R} \) is quite elegant and follows from the Artin-Wedderburn theorem (see \cite{vinberg1} for instance), which says that any finite-dimensional simple associative algebra over \( \mathbb{R} \) is isomorphic to a matrix algebra over a division algebra over \( \mathbb{R} \).
	
	There are only three types of finite-dimensional division algebras over \( \mathbb{R} \):
		 \( \mathbb{R} \) itself,
		 \( \mathbb{C} \) and
		the quaternions \( \mathbb{H} \).
		So, a finite-dimensional simple associative algebra over \( \mathbb{R} \) is isomorphic to a matrix algebra \( M_n(D) \), where \( D \) is one of these division algebras and \( n \geq 1 \). Let us show first that if $n\geq2$, the Koszul form of $M_n(D)$ is not positive definite. Indeed, in the three cases, one can construct a nilpotent matrix $A\not=0$ satisfying $A^2=0$ and hence $\mathrm{B}(A,A)=0$. For $n=1$, we have
		\[ \mathrm{B}(X,Y)=\begin{cases}
			XY\quad\mbox{if}\quad D=\R,\\
			2\mathfrak{R} (XY)\quad\mbox{if}\quad D=\mathbb{C},\\
			4\mathfrak{R} (XY)\quad\mbox{if}\quad D=\mathbb{H}.
		\end{cases} \]
	The Koszul form is positive definite if and only if $n=1$ and $D=\R$. This completes the proof.
	\end{proof}

Let us now study the general case.
Let $(\G, \bullet)$ be a left symmetric algebra with the associated Lie bracket $\br$. We denote by $\Li$ and $\Ri$ the left and right multiplication operators associated with $\bullet$, respectively.

Suppose that the Koszul form
$\mathrm{B}(X,Y)=\tr(\Li_{X\bullet Y})$ is  positive definite.    We have seen that $\mathrm{B}$ satisfies \eqref{6} and $(\G,\bullet,\mathrm{B})$ becomes a Hessian algebra. For any vector subspace $V\in\G$, $V^\perp$ denotes its orthogonal with respect to $\mathrm{B}$. 

Since $\mathrm{B}$ is nondegenerate there exists a non zero vector $H$ satisfying
$\mathrm{B}(X,H)=\tr(\Li_X)$,
for any $X\in\G$. Put $\h=H^\perp$.
From the definition of $\mathrm{B}$ we get
\begin{equation}\label{Hi}
	\mathrm{B}(X,H)=\tr(\Li_{X\bullet H})=\mathrm{B}(X\bullet H,H)=\mathrm{B}(H\bullet X,H),\quad X\in\G.
\end{equation}

\begin{pr}\label{hp}We have
	$\Li_H(\h)\subset \h$, $\Ri_H(\h)\subset \h$,  $\Ri_H$ is symmetric with respect to $B$ and
	$H\bullet H= H$.

\end{pr}
\begin{proof} The inclusions $\Li_H(\h)\subset \h$ and $\Ri_H(\h)\subset \h$ follow immediately from \eqref{Hi}. Moreover, the relation \eqref{6} gives
	\[ \mathrm{B}([X,Y],H)=\mathrm{B}(\Ri_HY,X)-
	\mathrm{B}(\Ri_HX,Y). \]
	But $\mathrm{B}([X,Y],H)=\tr(\Li_{[X,Y]})=0$  and $\Ri_H^*=\Ri_H$.
	Since $\Ri_H$ is symmetric and leaves $\h$ invariant then there exists $\mu$ such that $H\bullet H=\mu H$. Moreover, $\mathrm{B}(H,H)=\mathrm{B}(H\bullet H,H)=\mu \mathrm{B}(H,H)$ and hence $\mu=1$.
\end{proof}

We have $\G=\R H\oplus \h$ and, for any $X,Y\in\h$, there exists a unique $X\circ Y\in\h$ such that
\begin{equation*}\label{bullet} X\bullet Y=X\circ Y+\frac1\rho \mathrm{B}(X,Y)H. \end{equation*} where $\rho=\mathrm{B}(H,H)$. Define $A:\h\too\h$, $X\mapsto H\bullet X$ and $S:\h\too\h$, $X\mapsto X\bullet H$ ($S$ is symmetric) and denote by $\prs$ the restriction of $\frac1\rho \mathrm{B}$ to $\h$. The left multiplication operator  and the associator associated to $\circ$ will be denoted by $\Li^\circ$ and $\ass_\circ$, respectively. For any endomorphism $F$ of $\h$, $F^*$ is its adjoint with respect to $\prs$. To summarize, we have, for any $X,Y\in\h$,
\begin{equation}\label{product}
	X\bullet Y=X\circ Y+ \langle X,Y\rangle H,\; H\bullet X=AX,\; X\bullet H=SX\esp H\bullet H=H.	
\end{equation}Since, for any $X\in\h$,  $\mathrm{B}(X,H)=\tr(\Li_X)=0$, we deduce that $\tr(\Li_X^\circ)=0$.

\begin{pr}\label{h} Let $(\G,\bullet)$ be a $\mathrm{LSPK}$. With the notations above,
	for any $X,Y,Z\in\h$, 
	\begin{equation} \label{AS}\begin{cases}\langle X\circ Y-Y\circ X,Z\rangle=\langle Y\circ Z,X\rangle-\langle X\circ Z,Y\rangle,\\
			\mathrm{ass}_\circ(X,Y,Z)
			-\mathrm{ass}_\circ(Y,X,Z)=
			\left( \langle Y,Z\rangle SX-\langle X,Z\rangle SY\right),\\
			S([X,Y])=X\circ SY-Y\circ SX,\\
			A(X\circ Y)=AX\circ Y+X\circ AY-SX\circ Y,\\
			S=A+A^*-\mathrm{Id}_\h,\;
			[S,A]=S^2-S,\\
			[X,Y]=X\circ Y-Y\circ X,\quad\tr(\Li_X^\circ)=0.
		\end{cases}
	\end{equation}

\end{pr}

\begin{proof}We have, for any $U,V,W\in\G$,
	\[ \mathrm{B}(U\bullet V-U\bullet V,W)=\mathrm{B}(V\bullet W,U)-\mathrm{B}(U\bullet W,V). \]
	By using the splitting $\G=\h\oplus\R H$ and \eqref{bullet}, this relation gives:
	\begin{itemize}
		\item 
		For any $X,Y,Z\in\h$, 
		\[ \langle X\circ Y-Y\circ X,Z\rangle=\langle Y\circ Z,X\rangle-\langle X\circ Z,Y\rangle. \]
		
		\item For $X,Y\in\h$, 
		\[ \mathrm{B}(X\bullet Y-Y\bullet X,H)=\mathrm{B}(Y\bullet H,X)-\mathrm{B}(X\bullet H,Y). \]
		But $\mathrm{B}(X\bullet Y-Y\bullet X,H)=0$	and hence $S$ is symmetric.
		
		\item For any $X,Y\in\h$, 
		\[ \mathrm{B}(X\bullet H-H\bullet X,Y)=\mathrm{B}(H\bullet Y,X)-\mathrm{B}(X\bullet Y,H). \]	This can be written
		\[ \mathrm{B}((S-A)X,Y)=\mathrm{B}(AY,X)-\mathrm{B}(X,Y) \]and hence
		\[ S=A+A^*-\mathrm{Id}_\h. \]
		\item For any $X\in\h$, 
		\[ \mathrm{B}(X\bullet H-H\bullet X,H)=\mathrm{B}(H\bullet H,X)-\mathrm{B}(X\bullet H,H) \]	 holds.
	\end{itemize}
	On the other hand, for any $U,V,W\in \G$,
	\[ \ass(U,V,W)=\ass(V,U,W). \]
	Let us expand this relation by using \eqref{bullet}.		
	\begin{itemize}
		\item 
		For any $X,Y,Z\in\h$,
		\begin{align*}
			\mathrm{ass}(X,Y,Z)&=(X\circ Y+\langle X,Y\rangle H)\bullet Z-X\bullet (Y\circ Z+ \langle Y,Z\rangle H)\\
			&=\mathrm{ass}_\circ(X,Y,Z)+\left( \langle X\circ Y,Z\rangle-\langle X,Y\circ Z\rangle\right)H+\left(\langle X,Y\rangle A Z-\langle Y,Z\rangle SX \right),\\
			\mathrm{ass}(Y,X,Z)&=
			\mathrm{ass}_\circ(Y,X,Z)
			+\left( \langle Y\circ X,Z\rangle-\langle Y,X\circ Z\rangle \right)H+\left(\langle X,Y\rangle A Z-\langle X,Z\rangle SY \right)
		\end{align*}
		So
		\[ \begin{cases}
			\langle X\circ Y-Y\circ X,Z\rangle=\langle Y\circ Z,X\rangle-\langle X\circ Z,Y\rangle,\\
			\mathrm{ass}_\circ(X,Y,Z)
			-\mathrm{ass}_\circ(Y,X,Z)=
			\left( \langle Y,Z\rangle SX-\langle X,Z\rangle SY\right).
		\end{cases} \]
		\item For any $X,Y\in\h$,
		\begin{align*}
			\ass(X,Y,H)&=S(X\circ Y)+ \langle X,Y\rangle H     -X\circ SY- \langle X,SY\rangle H
		\end{align*}
		So
		\[ S([X,Y])=X\circ SY-Y\circ SX. \]
		
		\item For any $X,Y\in\h$,
		\begin{align*}
			\mathrm{ass}(X,H,Y)&,=(X\bullet H)\bullet Y-X\bullet (H\bullet Y)\\
			&=SX\circ Y+ \langle SX,Y\rangle H-X\circ AY-\langle X,AY\rangle H,\\
			\ass(H,X,Y)&=AX\circ Y+ \langle AX,Y\rangle H-A(X\circ Y)-
			\langle X,Y\rangle H\bullet H.
		\end{align*}So
		\[ SX\circ Y+A(X\circ Y)=AX\circ Y+X\circ AY\esp
		S=A+A^*-\mathrm{Id}_\h.
		\]
		
		\item For any $X,Y\in\h$,	\begin{align*}
			\ass(X,Y,H)&=S(X\circ Y)+ \langle X,Y\rangle H     -X\circ SY-\langle X,SY\rangle H
		\end{align*}
		So
		$S([X,Y])=X\circ SY-Y\circ SX.$
		\item For any $X\in\h$,	\begin{align*}
			\ass(X,H,H)&=S^2X-SX,\;\;
			\ass(H,X,H)=SAX-ASX.
		\end{align*}So
		$[S,A]=S^2-S.$ This completes the proof.\qedhere
	\end{itemize}
	
\end{proof}

The following two lemmas will be highly useful in completing our study.

\begin{Le}\label{useful} Let $(V,\prs)$ be an Euclidean vector space and $S,A$ two endomorphisms such that $S$ is symmetric 	and
	\[S=A+A^*-\mathrm{Id}_V,
	[S,A]=S^2-S.\]
	Then $V=V_1\oplus V_2$ where $V_2=V_1^\perp$, $V_1,V_2$ are invariant by $A$ and $S$ and 
	\[ S_{|V_1}=0,\; S_{|V_2}=\mathrm{Id}_{V_2},\;
	A_{|V_1}=B_1+\frac12\mathrm{Id}_{V_1},\; A_{|V_2}=B_2+\mathrm{Id}_{V_2} \]where $B_1:V_1\too V_1$ and $B_2:V_2\too V_2$ are skew-symmetric. 
\end{Le}
\begin{proof} Since $S$ is symmetric, we have $V=V_1\oplus V_2$ where $V_1=\ker S$ and $V_2=\mathrm{Im}S$. 
	The relation $[A^*,S]=S^2-S$ and its adjoint $[S,A]=S^2-S$ imply that  $A,A^*$ leaves invariant $V_1$ and $V_2$.

	In restriction to $V_1$, $A-\frac12\mathrm{Id}_{V_1}$ is skew-symmetric and hence $A_{|V_1}=B_1+
	\frac12\mathrm{Id}_{V_1}$ and $B_1:V_1\too V_1$ is skew-symmetric.
	
	In restriction to $V_2$, $S$ is invertible and  diagonalizable. So there exists an orthonormal basis $(e_1,\ldots,e_r)$ of $V_2$ such that $S(e_i)=\la_i e_i$.
	For any $i\in\{1,\ldots,r\}$,	
	\[ SA(e_i)-\la_iA(e_i)=(\la_i^2-\la_i)e_i. \]
	So
	\[ \sum_{j=1}^n(\la_j-\la_i) a_{ji}e_j= (\la_i^2-\la_i)e_i.\]This implies that
	$\la_i=1$ hence $S_{|V_2}=\mathrm{Id}_{V_2}$ and $A_{|V_2}=B_2+\mathrm{Id}_{V_2}$. This completes the proof.
\end{proof}	

\begin{Le}\label{useful1}
	Let $(\h,\circ,\prs)$ be a Hessian algebra. Suppose that there exists a skew-symmetric endomorphism $B$ such that, for any $X,Y\in\h$,
	\[ B(X\circ Y)=B(X)\circ Y+Y\circ B(X)+\frac12 X\circ Y. \]Then $\circ=0$.
\end{Le}

\begin{proof} Not first $\mu(A,B)=\tr(AB^t)$ defines on the Lie algebra $\mathrm{so}(\h,\prs)$  a scalar product  satisfying  
	\[ \mu([A,C],D)+\mu(C,[A,D])=0,\quad A,C,D\in\mathrm{so}(\h,\prs). \]
	
	Since $B$ is skew-symmetric, $\h=\ker B\oplus \mathrm{Im}B$ and  there exists $(\la_1,\ldots,\la_s)$ and an orthonormal basis $(e_1,f_1,\ldots,e_s,f_s)$ of $\mathrm{Im}B$ such that $B(e_j)=\la_j f_j$ and $B(f_j)=-\la_j e_j$, $j\in\{1,\ldots,s\}$. 
	
	Let $X\in\ker B$. Then $$[B,\Li_X]=\frac12\Li_X. $$ and hence $\tr(\Li_X)=0$. Since $B$ is skew-symmetric then $[B,\Li_X^*]=\frac12\Li_X^*$ and hence
	\[ [B,\Li_X-\Li_X^*]=\frac12(\Li_X-\Li_X^*). \]By applying $\mu$ to  we get that $\Li_X-\Li_X^*=0$ and hence $\Li_X$ is symmetric.
	
	On the other hand, fix $i\in\{1,\ldots,s\}$ and put $(X,Y)=(e_i,f_i)$,  $E=\Li_X-\Li_X^*$ and $F=\Li_Y-\Li_Y^*$. Then
	\[ [B,\Li_X]=\la \Li_Y+\frac12\Li_X\esp [B,\Li_Y]=-\la \Li_X+\frac12\Li_Y. \]This relation implies obviously that $\tr(\Li_X)=\tr(\Li_Y)=0$.
	Since $B$ is skew-symmetric, we deduce that
	\[ [B,E]=\la F+\frac12E\esp [B,F]=-\la E+\frac12 F. \]	
	By applying $\mu$, we get
	\[ \begin{cases}
		\frac12\mu(E,E)+\la\mu(E,F)=0,\\
		-\la \mu(E,F)+\frac12\mu(F,F)=0.
	\end{cases} \]	The discriminant of of this system is equal to $\frac14+\la^2\not=0$ we deduce that $E=F=0$. So far, we have shown that, for any $X\in\h$, $\Li_X$ is symmetric and $\tr(\Li_X)=0$. This property and the relation 
	\[ \langle X\circ Y-Y\circ X,Z\rangle =\langle Y\circ Z,X\rangle-\langle X\circ Z,Y\rangle \]
	implies that $\circ$ is commutative. But a commutative left symmetric product must be associative and hence, for any $X,Y\in\h$, $\Li_{X\circ Y}=\Li_X\circ \Li_Y$. This implies that $\tr(\Li_X^2)=0$ and since $\Li_X$ is symmetric, we get that $\Li_X=0$. This completes the proof.
\end{proof}

We have seen in Proposition \ref{h} that the study of left symmetric algebras with definite positive Koszul form reduces to the sturdy of Euclidean algebras $(\h,\circ,\prs)$ endowed with two endomorphism $A,S$ where $S$ is symmetric and the system \eqref{AS} holds.
Let $(\h,\circ,\prs)$ be a such algebra. According to Lemma \ref{useful},
$\h=\h_1\oplus\h_2$ such that $S_{|\h_1}=0$, $S_{|\h_2}=\mathrm{Id}_{\h_2}$, $A_{|\h_1}=B_1+\frac12\mathrm{Id}_{\h_1}$ and 
$A_{|\h_2}=B_2+\mathrm{Id}_{\h_2}$ with $B_1$ and $B_2$ are skew-symmetric.

For any $X,Y\in\h_1$ and $Z,T\in\h_2$, put
\[ X\circ Y=X\circ_1 Y+\om_{1}(X,Y)\esp Z\circ T=\om_2(Z,T)+Z\circ_2T \]where $X\circ_1Y,
\om_2(Z,T)\in\h_1$ and $Z\circ_2T,
\om_1(X,Y)\in\h_2$. From the third relation in \eqref{AS}, We deduce that $S([X,Y])=0$ and $S([Z,T])=[Z,T]$ and $\h_1$ and $\h_2$ are two subalgebras and $\om_1$ and $\om_2$ are symmetric.

On the other hand, for any $X\in\h_1$ and $Y\in\h_2$, from the third relation in \eqref{AS}, we get $S([X,Y])=X\circ Y$ and hence $X\circ Y\in \mathrm{Im}S=\h_2$. Furthermore,
\[ S([X,Y])=S(X\circ Y)-S(Y\circ X)=X\circ Y. \]ans since $S(X\circ Y)=X\circ Y$, we deduce that $Y\circ X\in\ker S=\h_1$.

Define $\rho_1:\h_1\too \mathrm{End}(\h_2)$ and $\rho_2:\h_2\too \mathrm{End}(\h_1)$ by putting
\[ \rho_1(X)(Y)=X\circ Y\esp \rho_2(Y)(X)=Y\circ X,\quad X\in\h_1, Y\in\h_2. \]
So far, we have shown that $\h=\h_1\oplus\h_2$,  there exists two products $\circ_i$ on $\h_i$ for $i=1,2$, $\om_1:\h_1\times\h_1\too\h_2$, $\om_2:\h_2\times\h_2\too\h_1$ symmetric, $\rho_1:\h_1\too \mathrm{End}(\h_2)$ and $\rho_2:\h_2\too \mathrm{End}(\h_1)$ such that the product $\circ$ is given by
\begin{equation}\label{match}
	X\circ Y=\begin{cases}
		X\circ_1 Y+\om_1(X,Y),\quad\mbox{if}\quad X,Y\in\h_1,\\
		\om_2(X,Y)+X\circ_2 Y,\quad\mbox{if}\quad X,Y\in\h_2,\\
		\rho_1(X)(Y)\quad\mbox{if}\quad X\in\h_1,Y\in\h_2,\\
		\rho_2(X)(Y)\quad\mbox{if}\quad X\in\h_2,Y\in\h_1.
\end{cases}\end{equation}$S_{|\h_1}=0$, $S_{|\h_2}=\mathrm{Id}_{\h_2}$, $A_{|\h_1}=B_1+\frac12\mathrm{Id}_{\h_1}$ and 
$A_{|\h_2}=B_2+\mathrm{Id}_{\h_2}$ with $B_1$ and $B_2$ are skew-symmetric. We denote by $\prs_i$ the restriction of $\prs$ to $\h_i$. For $X\in\h_i$, we have
\[ \tr(\Li_X)=\tr(\Li_X^{\circ_i})+\tr(\rho_i(X)). \]
On the other hand, it is to check that the first relation in \eqref{AS} is equivalent to
\begin{equation}\label{omega} \begin{cases}
		\langle X\circ_iY-Y\circ_iX,Z\rangle_i=\langle Y\circ_iZ,X\rangle_i-\langle X\circ_iZ,Y\rangle_i,\; X,Y,Z\in\h_i,\\
		\langle\om_1(X,Y),Z\rangle_2=\langle \rho_2(Z)(X),Y\rangle_1+\langle \rho_2(Z)(Y),X\rangle_1,\quad X,Y\in\h_1,Z\in\h_2, \\
		\langle\om_2(X,Y),Z\rangle_1=\langle \rho_1(Z)(X),Y\rangle_2+\langle \rho_1(Z)(Y),X\rangle_2,\quad X,Y\in\h_2,Z\in\h_1.
\end{cases} \end{equation}
This shows that $\om_1$ and $\om_2$ are defined by $\rho_1$ and $\rho_2$ via the metrics.

Next, we expand \eqref{AS} using \eqref{match} and, crucially, we find that $(\h_1, \circ_1, \prs_1, B_1)$ satisfies the conditions of Lemma \ref{useful1}, leading to the conclusion that $\circ_1 = 0$.
 The details of the computation are given in the Appendix.
\begin{pr}\label{mainpr} $(\h,\circ,\prs)$ satisfies \eqref{AS} if and only if $\rho_1:\h_1\too\mathrm{so}(\h_2,\prs_2)$ and $\rho_2:\h_2\too\mathrm{so}(\h_1,\prs_1)$ are two representations of Lie algebras, $\circ_1=0$, $B_1$ is skew-symmetric, $\tr(\rho_1(X))=0$ for any $X\in\h_1$ and the following systems hold:
	\begin{equation}\label{S1} \begin{cases}
			\langle\om_1(X,Y),Z\rangle_2=\langle \rho_2(Z)(X),Y\rangle_1+\langle \rho_2(Z)(Y),X\rangle_1,\quad X,Y\in\h_1,Z\in\h_2, \\
			\langle\om_2(X,Y),Z\rangle_1=\langle \rho_1(Z)(X),Y\rangle_2+\langle \rho_1(Z)(Y),X\rangle_2,\quad X,Y\in\h_2,Z\in\h_1.
	\end{cases} \end{equation}

	\begin{equation}\label{S2}  \begin{cases}
			\langle X\circ_2 Y-Y\circ_2 X,Z\rangle_2=
			\langle Y\circ_2 Z,X\rangle_2-\langle  X\circ_2 Z,Y\rangle_2,\\
			\ass_{\circ_2}(X,Y,Z)-\ass_{\circ_2}(Y,X,Z)=
			\left(\langle Y,Z\rangle_2 X-\langle X,Z\rangle_2 Y \right),\\
			B_2(X\circ_2 Y)=B_2(X)\circ_2Y+X\circ_2B_2(Y),\\
			\langle B_2X,Y\rangle_2=-\langle B_2Y,X\rangle_2,\tr(\Li_X^{\circ_2})=-\tr(\rho_2(X)),\quad X,Y,Z\in\h_2.
	\end{cases} \end{equation}
	
	\begin{equation}\label{S3}  \begin{cases}
			\rho_1(X)(Y\circ_2 Z)=Y\circ_2\rho_1(X)(Z)
			+\rho_1(X)(Y)\circ_2  Z-\rho_1(\rho_2(Y)(X))(Z)-\om_1(X,\om_2(Y,Z)),\quad X\in\h_1,Y,Z\in\h_2,\\
			\rho_2(\rho_1(Y)(X))(Z)+\om_2(X,\om_1(Y,Z))=0,\;X\in\h_2,Y,Z\in\h_1\\
			\rho_1(X)(\om_1(Y,Z))-\rho_1(Y)(\om_1(X,Z)) =0,\; X,Y,Z\in\h_1,\\
			\rho_2(X)(\om_2(Y,Z))-\rho_2(Y)(\om_2(X,Z)) +\om_2(X,Y\circ_2 Z)-\om_2(Y,X\circ_2 Z)-{\om_2([X,Y], Z)}=0,\;X,Y,Z\in\h_2,\\	
			\om_2(\rho_1(X)(Y),Z)+\om_2(Y,\rho_1(X)(Z))=0,\; X\in\h_1,Y,Z\in\h_2,\\
			X\circ_2(\om_1(Y,Z))=\om_1(\rho_2(X)(Y),Z)+
			\om_1(Y,\rho_2(X)(Z))-\langle Y,Z\rangle_1 X,\;X\in\h_2,Y,Z\in\h_1,\\
			[B_2,\rho_1(X)]=\rho_1(B_1(X))+\frac12\rho_1(X),\; X\in\h_1,\\
			[B_1,\rho_2(X)]=\rho_2(B_2(X)),\; X\in\h_2.	\end{cases} \end{equation}

\end{pr}
\begin{proof}
	See the Appendix.
\end{proof}
In conclusion, we have shown the following result which is our main result and gives a complete description of $\mathrm{LSPK}$.
\begin{theo}\label{theo} Let $(\h_1,\prs_1)$ be a Euclidean vector space, $(\h_2,\circ_2,\prs_2)$ a Euclidean algebra, $B_i$ is skew-symmetric endomorphism of $\h_i$, $i=1,2$, $\rho_1:\h_1\too\mathrm{End}(\h_2)$, $\rho_2:\h_2\too\mathrm{End}(\h_1)$ such that:
	\begin{enumerate}
		\item[$(i)$]  $\tr(\rho_1(X))=0$ and $[\rho_1(X),\rho_1(Y)]=0$, 
		\item[$(ii)$] $\rho_2(X\circ_2 Y-Y\circ_2 X)=
		[\rho_2(X),\rho_2(Y)]$ for any $X,Y\in\h_2$,
		\item[$(iii)$] $\om_1$ and $\om_2$ are defined by \eqref{S1} and \eqref{S2}-\eqref{S3} hold.
	\end{enumerate}	
	Then $\G=\h_1\oplus\h_2\oplus\R H$ endowed with the product $\bullet$ given by 
	\begin{equation*}\label{matchbis}
		X\bullet Y=\begin{cases}
			\om_1(X,Y)+\langle X,Y\rangle_1 H,\quad X,Y\in\h_1,\\
			\om_2(X,Y)+X\circ_2 Y+\langle X,Y\rangle_2 H,\quad X,Y\in\h_2,\\
			\rho_1(X)(Y),\quad X\in\h_1,Y\in\h_2,\\
			\rho_2(X)(Y),\quad X\in\h_2,Y\in\h_1.
		\end{cases}\begin{cases}
			H\bullet X=B_1(X)+\frac12 X,\; X\bullet H=0,\; X\in\h_1,\\
			H\bullet X=B_2(X)+ X,\; X\bullet H=X,\; X\in\h_2,\\
			H\bullet H=H,`
	\end{cases}\end{equation*} is a $\mathrm{LSPK}$ and the Koszul form $\mathrm{B}$ is given by $\mathrm{B}(\h_1,\h_2)=\mathrm{B}(H,\h_1)=\mathrm{B}(H,\h_2)=0,$
	\[  \mathrm{B}(X,Y)=\rho\langle X,Y\rangle_i,\quad X,Y\in\h_i,\; i=1,2\esp \mathrm{B}(H,H)=\rho,
	\]where $\rho=\left(\frac12\dim\h_1+\dim\h_2+1\right)$.  Moreover, all $\mathrm{LSPK}$ are obtained in this way.
\end{theo}

\begin{remark} When $\rho_1=0$, $\om_2=0$ and \eqref{S2} and \eqref{S3} reduce to
\[ X\circ_2(\om_1(Y,Z))=\om_1(\rho_2(X)(Y),Z)+
\om_1(Y,\rho_2(X)(Z))-\langle Y,Z\rangle_1 X\esp 
[B_1,\rho_2(X)]=\rho_2(B_2(X)),\;\;X\in\h_2,Y,Z\in\h_1. \]	
\end{remark}

This theorem has an important corollary.

\begin{co}\label{ex}\begin{enumerate}
		\item Let $(\h\prs)$ be a Euclidean vector space of dimension $n$  and $D$ a skew-symmetric endomorphism on $\h$. 
		Then $\G=\h\oplus\R H$ endowed with $\bullet$ given by
		\begin{equation}\label{1} X\bullet Y= \langle X,Y\rangle H,\; H\bullet X=\frac12X+D(X),\; X\bullet H=0,\;H\bullet H=H,\; X,Y\in\h, \end{equation}is a $\mathrm{LSPK}$ and its Koszul form 
		$\mathrm{B}$ is given by
		\[ \mathrm{B}(X,Y)=\left(\frac{n}2+1\right)\langle X,Y\rangle,\;\mathrm{B}(X,H)=0,\; \mathrm{B}(H,H)=\frac{n}2+1,\;\; X,Y\in\h. \]
		\item Let $(\h,\circ,\prs)$ be a Euclidean algebra of dimension $n$ such that,  for any $X,Y\in\h$, $\tr(\Li_X)=0$ and
		\[ \begin{cases}
			\langle X\circ Y-Y\circ X,Z\rangle=
			\langle Y\circ Z,X\rangle-\langle  X\circ Z,Y\rangle,\\
			\ass(X,Y,Z)-\ass(Y,X,Z)=
			\left(\langle Y,Z\rangle X-\langle X,Z\rangle Y \right),
		\end{cases} \]
		 and $D$ is a skew-symmetric derivation of $(\h,\circ)$. Then $\G=\h\oplus\R H$ endowed with $\bullet$ given by
		\begin{equation}\label{2} X\bullet Y=X\circ Y+ \langle X,Y\rangle H,\; H\bullet X=X+D(X),\; X\bullet H=X,\;H\bullet H=H,\; X,Y\in\G, \end{equation}is a $\mathrm{LSPK}$ and its Koszul form 
		$\mathrm{B}$ is given by
		\[ \mathrm{B}(X,Y)=(n+1)\langle X,Y\rangle,\;\mathrm{B}(X,H)=0,\; \mathrm{B}(H,H)=n+1,\;\; X,Y\in\h. \]
	\end{enumerate}
	
\end{co}

\section{A new class of non-associative algebras: $k$-Hessian algebras}\label{section4}

Theorem \ref{theo}  shows that all $\mathrm{LSPK}$ can be constructed from a class of algebras that, to our knowledge, is new and of independent interest.

A $k$-Hessian algebra is an algebra $(\h, \circ)$ equipped with a scalar product $\langle \cdot, \cdot \rangle$, such that for any $X, Y, Z \in \h$:
\begin{equation}\label{kh}
	\begin{cases} 
		\langle X \circ Y - Y \circ X, Z \rangle = \langle Y \circ Z, X \rangle - \langle X \circ Z, Y \rangle, \\
		\operatorname{ass}(X, Y, Z) - \operatorname{ass}(Y, X, Z) = k\left( \langle X, Z \rangle Y - \langle Y, Z \rangle X \right).
	\end{cases}
\end{equation}

It is important to note that a $k$-Hessian algebra is Lie-admissible, meaning the bracket $[X, Y] = X \circ Y - Y \circ X$ satisfies the properties of a Lie bracket. Furthermore, if $G$ is a connected Lie group with Lie algebra $(\G, \br)$, the pair $(\prs, \circ)$ defines a left-invariant metric $h$ on $G$ and a torsion-free connection $\nabla$, such that $(h, \nabla)$ satisfies the Codazzi equation \eqref{codazzi} and the curvature $R^\na$ of $\na$ satisfies
\[
R^\nabla(X, Y) = k X \wedge Y.
\]
We refer to the structure $(G, h, \nabla)$ as a $k$-Hessian Lie group, establishing a correspondence between $k$-Hessian Lie algebras and $k$-Hessian Lie groups. In particular, if $\nabla$ is the Levi-Civita connection of $(G, h)$, we obtain an important subclass of $k$-Hessian algebras: the Lie algebras associated with Lie groups that have a left-invariant Riemannian metric of constant sectional curvature $k$.

For $k > 0$, the only connected and simply connected Lie group carrying a left-invariant metric with constant sectional curvature $k$ is $SU(2)$ (see \cite{wallach}). 

For $k < 0$, Milnor in \cite{Milnor} provided a class of Euclidean Lie algebras with constant sectional curvature $k$. Indeed, let $(\h,  \prs)$ be  Euclidean vector space and $\bold{h}\in\h\setminus\{0\}$. The bracket on $\h$   defined by:
\[
[X, Y] = (X \wedge Y)(\mathbf{h}) 
\]is a Lie bracket and the Levi-Civita product of $(\G,\br,\prs)$ is given by
\[ X \circ Y = \langle X, Y \rangle \mathbf{h} - \langle Y, \mathbf{h} \rangle X,\;\quad X,Y\in\h. \]Furthermore, the associator satisfies the relation: 
\[ \ass(X,Y,Z)-\ass(X,Y,Z)=-|\mathbf{h}|^2 (X\wedge Y)(Z),\quad X,Y,Z\in\h. \]
Thus $(\h,\circ,\prs)$ is $-|\mathbf{h}|^2$-Hessian algebra. We refer to $(\h,\circ,\prs)$ as a Milnor algebra.
Note that $\Li_{\mathbf{h}}=0$ and, in fact, a $k$-Hessian algebra with $k<0$ and having a non zero vector $u$ satisfying $\Li_u=0$ is a Milnor algebra.

\begin{theo}\label{milnor}
	Let $(\h,\circ,\prs)$ be a $k$-Hessian algebra such that $k<0$ and there exists $u\not=0$ such that $\Li_u=0$. Then there exists $\bold{h}=\mu u$ such that, for any $X,Y\in\h$,
	\[ X\circ Y=\langle X,Y\rangle \bold{h}-\langle \bold{h},Y\rangle X. \]
	
\end{theo}

\begin{proof} Denote by $\br$ the Lie bracket associated to $\circ$. We have, for any $X,Y\in\h$,
	\[ \Li_{[X,Y]}-[\Li_X,\Li_Y]= k X\wedge Y \]where $X\wedge Y$ is the skew-symmetric endomorphism given by
	\[ (X\wedge Y)(Z)=\langle X,Z\rangle Y-\langle Y,Z\rangle X. \]
	We can suppose that $|u|=1$. Since $\Li_u=0$, for any $X\in\h$,
	\begin{equation}\label{U} \Li_{X\circ u}= k X\wedge u. \end{equation}
	This relation implies that  $\ker\Ri_u=\R u$ and  $\h=\R u\oplus \mathrm{Im}\Ri_u$. So,  for any $X\in\h$, $\Li_X$ is skew-symmetric. We deduce that $u^\perp=\mathrm{Im}\Ri_u:=\h_1$.
	
	If $\dim\h=2$, choose a unit vector $e$ such that $\langle u,e\rangle=0$. We have $e\circ u=\la e$ and $e\circ e=-\la  u$. So
	\begin{align*}
		\Li_{e\circ u}e=	\la e\circ e=k (e\wedge u)(e)=ku
	\end{align*}and hence
	$\la^2=-k.$
	We can choose $u$ such that
	$e\circ u=-\sqrt{|k|} e$ and the vector $\bold{h}= \sqrt{|k|} u$ satisfies the desired relation.

	Suppose $\dim\h\geq3$ and choose $(e_2,\ldots,e_n)$  an orthonormal basis of $\h_1$ and $(w_2,\ldots,w_n)$ its image by $\Ri_u$, i.e., $w_i=e_i\circ u$. According to \eqref{U}, for any $i,j\in\{2,\ldots,n\}$,
	\begin{align}
		w_i\circ u&=\Li_{w_i}u=k (e_i\wedge u)(u)=-ke_i,\nonumber\\\label{20}
		e_i\circ e_j&=-\frac1k \Li_{w_i\circ u}(e_j)=-\langle w_i,e_j\rangle u,\\ w_i\circ w_j&=k(e_i\wedge u)(w_j)=k\langle e_i,w_j\rangle u.\nonumber
	\end{align}
	Now, it is easy to check that for any skew-symmetric endomorphism $A$ and for any $X,Y$,
	\[ [A,X\wedge Y]=(AX)\wedge Y+X\wedge AY. \]
	By using this relation, we get for any $i\not=j$,
	\begin{align*}
		k w_i\wedge w_j&=-[\Li_{w_i},\Li_{w_j}]
		=-k[\Li_{w_i}, e_j\wedge u]
		=-k\Li_{w_i}(e_j)\wedge u-ke_j\wedge \Li_{w_i}(u)
		=k^2 e_j\wedge e_i.
	\end{align*}So $w_i\wedge w_j=-ke_i\wedge e_j$. This relation implies that,  for any $l\notin\{i,j\}$,  
	\[ \langle w_i,e_l\rangle w_j- \langle w_j,e_l\rangle w_i=0 \]	and hence
	\[ \langle w_i,e_l\rangle=\langle w_j,e_l\rangle=0. \]Fix $i\not=j$. We have
	\[ w_i=a e_i+b e_j \esp w_j=c e_i+d e_j.\]
	Put $a=\langle w_i,e_i\rangle$, $b=\langle w_i,e_j\rangle$, $c=\langle w_j,e_i\rangle$ and $d=\langle w_j,e_j\rangle$. Note first that the relation $w_i\wedge w_j=-ke_i\wedge e_j$ implies
	\[ ad-bc=-k. \]Moreover, according to \eqref{20},
	\[ e_i\circ e_i=-au,\; e_i\circ e_j=-bu,\;e_j\circ e_i=-cu\esp e_j\circ e_j=-du. \]We get
	\begin{align*}
		\Li_{w_i}e_i&=-(a^2+bc)u=k(e_i\wedge u)(e_i)=ku,\\
		\Li_{w_i}e_j&=-(ab+bd)u=k(e_i\wedge u)(e_j)=0,\\
		\Li_{w_j}e_i&=-(ac+dc)u=k(e_j\wedge u)(e_i)=0,\\
		\Li_{w_j}e_j&=-(cb+d^2)u=k(e_i\wedge u)(e_i)=ku.
	\end{align*}Thus
	\[ a^2+bc=cb+d^2=ad-bc=-k,b(a+d)=c(a+d)=0. \]
	If $a+d=0$ then $cb+d^2=-d^2-bc=-k$ which is impossible so $b=c=0$ and $a^2=d^2=-k$. By replacing $u$ by $-u$ if necessary, we get
	\[ e_i\wedge e_i=e_j\circ e_j=-\sqrt{|k|}u,\; e_i\circ e_j=e_j\circ e_i=0\esp e_i\circ u= \sqrt{|k|}e_i. \]By taking $\bold{h}=-\sqrt{|k|} u$ we get the desired result.
\end{proof}

Another important class of $k$-Hessian algebras is formed by $k$-Hessian commutative algebras. Let $(\h,\circ,\prs)$ be a $k$-Hessian algebra such that $\circ$ is commutative. Then the first relation in \eqref{kh} equivalent to $\Li_X$ being symmetric for any $X\in\h$ and the second relations is equivalent to
\[ [\Li_X,\Li_Y]=-k X\wedge Y,\; X,Y\in\h. \]
In dimension 2, a $k$-Hessian algebra such that, for any $X$,  $\Li_X=0$ is either a Milnor algebra or a commutative $k$-Hessian algebra.
\begin{pr}\label{kdim2} Let $(\h,\circ,\prs)$ be a $k$-Hessian algebra such that, for any $X\in\h$, $\tr(\Li_X)=0$. Then $k\leq0$ and if $k=0$ then $\circ$ is trivial. If $k<0$ then there exists an orthonormal basis $(e_1,e_2)$ of $\h$ such that one of the following situations occurs:
	\[ \begin{cases}
		e_1\circ e_1=-\sqrt{|k|}\cos(\theta)e_2,\\
		e_1\circ e_2=\sqrt{|k|}\cos(\theta)e_1,\\
		e_2\circ e_1=-\sqrt{|k|}\sin(\theta)e_2,\\
		e_2\circ e_2=\sqrt{|k|}\sin(\theta) e_1,
	\end{cases}\quad\mbox{or}\quad\begin{cases}
		e_1\circ e_1=- y e_1+be_2,\\
		e_1\circ e_2=be_1+ y e_2,\\
		e_2\circ e_1= b e_1+ye_2,\\
		e_2\circ e_2=y e_1-b e_2,\\
		b=\frac{\sqrt{|k|}}{\sqrt{2}}\cos(\theta),
		y=\frac{\sqrt{|k|}}{\sqrt{2}}\sin(\theta).
	\end{cases} \]
\end{pr}

\begin{proof}
	We choose an orthonormal basis $(e_1,e_2)$ of $\h$ and put
	\[ \Li_{e_1}=\left(\begin{matrix}
		a&b\\c&-a
	\end{matrix}  \right)\esp \Li_{e_2}=\left(\begin{matrix}
		x&y\\z&-x
	\end{matrix}  \right).  \] The metric is Hessian if and only if $c=2x-b$ and $z=-(2a+y)$. Now the relation
	\[ \ass(X,Y,Z)-\ass(Y,X,Z)=k\left(\langle X,Z\rangle Y- \langle Y,Z\rangle X \right) \]is equivalent to
	\[ \begin{cases}
		ab+xy=0,\\
		bx-ay +b^2+ y^2 + k=0,\\
		6a^2 +5ay - 5bx + 6x^2 +b^2+ y^2 +k=0.
	\end{cases} \]
	If we take the difference between the third relation and the second one we get
	\[ 6(a^2+ay-bx+x^2)=0. \]But $ay-bx=b^2+ y^2 + k$ and hence
	\[ a^2+x^2+b^2+y^2=-k. \]
	Then $k\leq 0$ and if $k=0$, $a=x=b=y=0$ and $\circ=0$.
	
	Suppose that $k<0$. Then $(b,y)\not=(0,0)$ and from the relation $ab+xy=0$, we deduce that there exists $\mu$ such that $x=\mu b$ and $a=-\mu y$. So
	\[ \begin{cases}
		\mu b^2+\mu y^2 +b^2+ y^2 + k=0,\\
		6\mu^2y^2 -5\mu y^2 - 5\mu b^2 + 6\mu^2b^2 +b^2+ y^2 +k=0.
	\end{cases} \]Thus
	\[ (1+\mu)(b^2+y^2)=-k\esp (\mu^2-\mu)(y^2+b^2)=0. \]
	So $\mu\in\{0,1\}$, $b=\frac{\sqrt{|k|}}{\sqrt{1+\mu}}\cos\theta$ and
	$y=\frac{\sqrt{|k|}}{\sqrt{1+\mu}}\sin\theta$. 
\end{proof}

If we drop the condition $\tr(\Li_X)=0$, we can find examples of $k$-Hessian algebras with $k>0$ as illustrated by the following example. We give also an example of $-1$-Hessian commutative algebra of dimension 3.

\begin{exem}\begin{enumerate}
		\item 
		Consider $\R^2$ endowed with the metric and the product $\circ$ given by
		\[ \prs=\left[\begin{array}{cc}
			\lambda  & 0 
			\\
			0 & \mu  
		\end{array}\right],\;\begin{cases} e_1\circ e_1=\frac{(y+2)\la}{2\mu} e_2,e_2\circ e_2=ye_2,\\
			e_1\circ e_2=(\frac{y}{2}-1)e_1,\;e_2\circ e_1=\frac{y}{2}e_1,\;\mu=\frac1k(\frac14y^2-1),\quad k|y|>2.
		\end{cases}
		\]
		Then $(\R^2,\circ,\prs)$ is a $k$-Hessian algebra.
		\item Consider $\R^3$ endowed with its canonical scalar product and the product $\circ$ given by
		\[ e_1\circ e_2=e_2\circ e_1=e_3, e_1\circ e_3=e_3\circ e_1=e_2,e_1\circ e_3=e_3\circ e_1=e_2. \]
		Then $(\R^3,\circ,\prs)$ is a $-1$-Hessian commutative algebra.

	\item	By using  Milnor algebras introduced in Section \ref{section4}, we can build a large class of $\mathrm{LSPK}$.

			Let $(\h,\prs)$ be a Euclidean vector space and $\bold{h}$ is a unit vector. The product
			\[ X\circ Y=\langle X,Y\rangle \bold{h}-\langle Y,\bold{h}\rangle X \]defines on $\h$ a structure of $-1$-Hessian algebra. Moreover, one can see easily that $D$ is a derivation of $(\h,\circ)$ if and only if $D$ is skew-symmetric and $D(\bold{h})=0$. By using Corollary \ref{ex}, we get that $\h\oplus\R H$ endowed with the product $\bullet$ given by \eqref{2} is a $\mathrm{LSPK}$.

	\end{enumerate}
\end{exem}

The discovery of $k$-Hessian algebras, as the infinitesimal counterpart of $k$-Hessian Lie groups, represents an important consequence of our study. This naturally leads to a meaningful generalization of Hessian manifolds, introducing the concept of $k$-Hessian manifolds, as outlined in the introduction.

\section{Left symmetric algebras with positive definite Koszul form of dimension $\leq$ 3 and some examples of dimension 4 and 5 }\label{section6}

In this section, we give all $\mathrm{LSPK}$ of dimension $\leq$ 3 and some examples of dimension 4 and 5.

The classification of 2-dimensional $\mathrm{LSPK}$ follows directly from Theorem \ref{theo}.
\begin{pr} Let $(\G,\bullet)$ be a 2-dimensional $\mathrm{LSPK}$. Then $(\G,\bullet)$ is either isomorphic to $\R^2$ with its canonical associative product or there exists a basis $(e,H)$ of $\G$ such that
		 $e\bullet e= H,\; H\bullet e=\frac12e, H\bullet H=H.$ In this case, the matrix of the Koszul form is $\frac32\mathrm{I}_2$.

\end{pr}
The situation in dimension 3 is more intricate.
\begin{pr}\label{dim3} Let $(\G,\bullet)$ be a 3-dimensional $\mathrm{LSPK}$. Then Then $(\G,\bullet)$ is either isomorphic to $\R^3$ with its canonical associative product or there exists a basis $(e_1,e_2,H)$ of $\G$ such that one of the following cases holds:
	\begin{enumerate}
		
		\item \[ \begin{cases}
			e_1\bullet e_1=e_2\bullet e_2= H,\;e_1\bullet e_2=e_2\bullet e_1=0,\;
			H\bullet e_1=\la e_2+\frac12 e_1,\\H\bullet e_2=-\la e_1+\frac12 e_2,\;
			e_1\bullet H=e_2\bullet H=0, H\bullet H=H,\;\la\in\R.
		\end{cases} \]The matrix of the Koszul form is $2\mathrm{I}_3$.
		\item \[ \begin{cases}
			e_1\bullet e_1=-\cos(\theta)e_2+H,\;
			e_1\bullet e_2=\cos(\theta)e_1,\;
			e_2\bullet e_1=-\sin(\theta)e_2,\\
			e_2\bullet e_2=\sin(\theta) e_1 +H,\;
			H\bullet e_1=e_1\bullet H=e_1, H\bullet e_2=e_2\bullet H=e_2,\; H\bullet H=H.
		\end{cases} \]The matrix of the Koszul form is $3\mathrm{I}_3$.

	\item \[\begin{cases}
		e_1\bullet e_1=-2ae_2+ H,e_1\bullet e_2=0,e_2\bullet e_1=-a e_1,\\ e_2\bullet e_2= a e_2+H,H\circ e_1=\frac12 e_1, H\circ e_2=e_2,\\ e_1\circ H=0, e_2\circ H=e_2, H\circ H=H,\quad a=\pm\frac1{\sqrt{6}}.
	\end{cases}\]The matrix of the Koszul form is $\frac52\mathrm{I}_3$.
	
	\end{enumerate}
	
\end{pr}

\begin{proof} According to Theorem \ref{theo}, $\G=\h_1\circ\h_2\oplus \R H$ with data $(\circ_1,\circ_2,B_1,B_2,\rho_1,\rho_2)$ satisfying the conditions in the theorem. There are three possibilities:
	\begin{enumerate}
		\item $\dim\h_1=2$ and $\h_2=0$. We apply Corollary \ref{ex} to get the first case.
		\item $\dim\h_2=2$ and $\h_1=0$. We apply Proposition \ref{kdim2} and Corollary \ref{ex} to get two algebras one of them is commutative and by Theorem \ref{ass}   is isomorphic to $\R^3$. The other gives the second case. 
		\item  $\dim\h_1=\dim\h_2=1$. 
		Put $\h_1=\R e_1$ and $\h_2=\R e_2$ with $|e_1|=|e_2|=1$. $B_1=B_2=0$, $\circ_1=0$ and $e_2\circ_2 e_2=ae_2$, $\rho_1=0$ and $\rho_2(e_2)=-a\mathrm{id}_{\h_1}$. We have $\om_2=0$ and $\om_1(e_1,e_1)=-2a e_2.$
		The relation
		\[ 	X\circ_2(\om_1(Y,Z))=\om_1(\rho_2(X)(Y),Z)+
		\om_1(Y,\rho_2(X)(Z)){-}\langle Y,Z\rangle_1 X \]
		gives $-2a^2=-2a(-2a)-1$ and hence
		 $1=6a^2$. \qedhere
	\end{enumerate}
	\end{proof}

To get examples in dimension 4 and 5, let us solve the systems \eqref{S1}-\eqref{S3} when $\dim\h_1\in\{2,3\}$,  $\dim\h_2=1$ and $B_1\not=0$. Put $\h_2=\R f$,  $A=\rho_2(f)$ and $f\circ_2 f=af$.  We have $B_2=0$, $\rho_1=0$,  $\om_2=0$,  $\om_1=\om f$ and $B_1$ skew-symmetric.   The triple $(\om,B_1,A)$ satisfies
	\begin{equation}\label{sys}\begin{cases}
			\tr(A)=-a,\; \;  \om(X,Y)=\langle A(X),Y\rangle_1+\langle A(Y),X\rangle_1,\quad [B_1,A]=0,\\
			a\om(X,Y)=\om(A(X),Y)+\om(A(Y),X)-\langle X,Y\rangle_1, \quad X,Y\in\h_1.
	\end{cases} \end{equation}
	
	\begin{itemize}
		
		\item $\dim\h_1=2$ and $B_1\not=0$. There exists an orthonormal basis $(e_1,e_2)$ of $\h_1$ such that $B_1$ and $A$ are given by their matrices
		\[ B_1=\left(\begin{matrix}
			0&\la\\-\la&0
		\end{matrix} \right)\esp A=\left(\begin{matrix}
			\al &\be\\-\be&\al
		\end{matrix} \right),\; \la>0,\; a=-2\al. \]
		We have
		\[ \om(e_1,e_1)=\om(e_2,e_2)=2\al\esp \om(e_1,e_2)=0. \]
		The last equation in \eqref{sys} is equivalent to $8\al^2-1=0$.
		
	We get a 4-dimensional $\mathrm{LSPK}$ where the non vanishing products are given by
	\[ \begin{cases}
		e_1\bullet e_1=e_2\bullet e_2=2\al f+H,\\
		f\bullet f=-2\al f+H,\; f\bullet e_1=\al e_1-\be e_2,f\bullet e_2=\be e_1+\al e_2,\\
		H\bullet e_1=-\la e_2+\frac12 e_1, H\bullet e_2=\la e_1+\frac12 e_2, H\bullet f=f\bullet H=f, H\bullet H=H,\; \al^2=\frac18,\be\in\R,\;\la>0.
		\end{cases} \]The matrix of the Koszul form is $3\mathrm{I}_4$.

		\item $\dim\h_1=3$ and $B_1\not=0$.  There exists an orthonormal basis $(e_1,e_2,e_3)$ of $\h_1$ such that $B_1$ and $A$ are given by their matrices
		\[ B_1=\left(\begin{matrix}
			0&\la&0\\-\la&0&0\\0&0&0
		\end{matrix} \right)\esp A=\left(\begin{matrix}
			\al &\be&0\\-\be&\al&0\\
			0&0&\ga
		\end{matrix} \right),\; \la>0,\; a=-(2\al+\ga). \]
		We have
		\[ \om(e_1,e_1)=\om(e_2,e_2)=2\al,\;\om(e_3,e_3)=2\ga\esp \om(e_1,e_2)=\om(e_1,e_3)=\om(e_2,e_3)=0. \]
		\end{itemize}	
	The last equation in \eqref{sys} is equivalent to
\[ 8\al^2+2\al\ga-1=0\esp 6\ga^2+4\al\ga-1=0. \]
\[ (\ga,\al)=\left(\pm \frac{1}{\sqrt{3}},\pm \frac{-3}{4\sqrt{3}}\right)\quad\mbox{or}\quad 
(\ga,\al)=\left(\pm \frac{1}{\sqrt{10}},\pm \frac{1}{\sqrt{10}}\right).\]	
We get  a 5-dimensional $\mathrm{LSPK}$ where the non vanishing products are given by
\[ \begin{cases}
	e_1\bullet e_1=e_2\bullet e_2=2\al f+H,e_3\bullet e_3=2\ga f+H\\
	f\bullet f=-(2\al+\ga) f+H,\; f\bullet e_1=\al e_1-\be e_2,f\bullet e_2=\be e_1+\al e_2,f\bullet e_3=\ga e_3,\\
	H\bullet e_1=-\la e_2+\frac12 e_1, H\bullet e_2=\la e_1+\frac12 e_2, H\bullet e_3=\frac12 e_3, H\bullet f=f\bullet H=f, H\bullet H=H,\;\\ \be\in\R,\;(\al,\ga)\in\left\{\left(\pm \frac{1}{\sqrt{3}},\pm \frac{-3}{4\sqrt{3}}\right),\left(\pm \frac{1}{\sqrt{10}},\pm \frac{1}{\sqrt{10}}\right)\right\},\;\la>0.
\end{cases} \]	The matrix of the Koszul form is $\frac72\mathrm{I}_5$.

\section{Appendix}\label{section7}

\subsection{Proof of Proposition \ref{mainpr}}
Let $(\h,\circ,\prs)$ be an algebra endowed with a scalar product and $A,S$ two endomorphisms of $\h$ such that $S$ is symmetric and the system \eqref{AS} holds. We have shown that $\h=\h_1\oplus\h_2$,  there exists a product $\circ_i$ and a metric $\prs_i$ on $\h_i$ for $i=1,2$ satisfying \eqref{omega}, $\rho_1:\h_1\too \mathrm{End}(\h_2)$,  $\rho_2:\h_2\too \mathrm{End}(\h_1)$ and $\om_1:\h_1\times\h_1\too\h_2$, $\om_2:\h_2\times\h_2\too\h_1$  given by \eqref{omega}  such that the product $\circ$ is given by \ref{match},
$S_{|\h_1}=0$, $S_{|\h_2}=\mathrm{Id}_{\h_2}$, $A_{|\h_1}=B_1+\frac12\mathrm{Id}_{\h_1}$ and 
$A_{|\h_2}=B_2+\mathrm{Id}_{\h_2}$ with $B_1$ and $B_2$ are skew-symmetric. Moreover, \eqref{AS} reduces to
\begin{equation}\label{dAS}
	\begin{cases}
		\mathrm{ass}_\circ(X,Y,Z)
		-\mathrm{ass}_\circ(Y,X,Z)=
		\left( \langle Y,Z\rangle SX-\langle X,Z\rangle SY\right),\\
		A(X\circ Y)=AX\circ Y+X\circ AY-SX\circ Y.
	\end{cases}
\end{equation}

Let us expand the first relation in \eqref{dAS}. We distinguish many cases.
\begin{itemize}
	\item 
	For any $X,Y,Z\in\h_1$,
	\begin{align*}
		\ass_\circ(X,Y,Z)&=(X\circ Y)\circ Z-X\circ(Y\circ Z)\\
		&=(X\circ_1 Y)\circ_1 Z+\om_1(X\circ_1 Y,Z)+\rho_2(\om_1(X,Y))(Z)-X\circ_1(Y\circ_1 Z)-\om_1(X,Y\circ_1 Z)-\rho_1(X)(\om_1(Y,Z))\\
		&=\ass_{\circ_1}(X,Y,Z)+\rho_2(\om_1(X,Y))(Z)
		-\om_1(X,Y\circ_1 Z)-\rho_1(X)(\om_1(Y,Z))
	\end{align*}So
	\[ \begin{cases}
		\ass_{\circ_1}(X,Y,Z)=\ass_{\circ_1}(Y,X,Z),\\
		\rho_1(X)(\om_1(Y,Z))-\rho_1(Y)(\om_1(X,Z)) +\om_1(X,Y\circ_1 Z)-\om_1(Y,X\circ_1 Z)
		-\om_1([X,Y],Z)=0,\\ X,Y,Z\in\h_1,
	\end{cases} \]In the same way, we get
	\[ \begin{cases}
		\ass_{\circ_2}(X,Y,Z)-\ass_{\circ_2}(Y,X,Z)=
		\left(\langle Y,Z\rangle X-\langle X,Z\rangle Y \right),\\
		\rho_2(X)(\om_2(Y,Z))-\rho_2(Y)(\om_2(X,Z)) +\om_2(X,Y\circ_2 Z)-\om_2(Y,X\circ_2 Z)-\om_2([X,Y],Z)=0,\\ X,Y,Z\in\h_2.
	\end{cases} \]
	
	\item For $X,Y\in\h_1$ and $Z\in\h_2$, we have
	\begin{align*}
		\ass_{\circ}(X,Y,Z)&=(X\circ Y)\circ Z-X\circ(Y\circ Z)\\
		&=\rho_1(X\circ_1 Y)(Z)+\om_1(X,Y)\circ_2 Z+\om_2(\om_1(X,Y),Z)-\rho_1(X)\circ\rho_1(Y)(Z).
	\end{align*}So $\rho_1$ is a representation of Lie algebras. In the same way, we get that $\rho_2$ is a also a representation of Lie algebras.
	
	\item For $X,Z\in\h_1$ and $Y\in\h_2$, we have
	\begin{align*}
		\ass_{\circ}(X,Y,Z)&=(X\circ Y)\circ Z-X\circ(Y\circ Z)\\
		&=\rho_2(\rho_1(X)(Y))(Z)-X\circ_1\rho_2(Y)(Z)-\om_1(X,\rho_2(Y)(Z)),\\
		\ass_{\circ}(Y,X,Z)&=(Y\circ X)\circ Z-Y\circ(X\circ Z)\\
		&=\rho_2(Y)(X)\circ_1 Z+\om_1(\rho_2(Y)(X),Z)
		-\rho_2(Y)(X\circ_1 Z)-Y\circ_2(\om_1(X,Z))-\om_2(Y,\om_1(X,Z))	
	\end{align*}So
	\[\begin{cases} Y\circ_2(\om_1(X,Z))=\om_1(\rho_2(Y)(X),Z)+\om_1(X,\rho_2(Y)(Z))-\langle X,Z\rangle Y,\\ 
		\rho_2(Y)(X\circ_1 Z)=X\circ_1\rho_2(Y)(Z)+
		\rho_2(Y)(X)\circ_1 Z
		-\om_2(Y,\om_1(X,Z))-\rho_2(\rho_1(X)(Y))(Z).\end{cases} \]

	\item For $X,Z\in\h_2$ and $Y\in\h_1$, we have
	\begin{align*}
		\ass_{\circ}(X,Y,Z)&=(X\circ Y)\circ Z-X\circ(Y\circ Z)\\
		&=\rho_1(\rho_2(X)(Y))(Z)-X\circ_2\rho_1(Y)(Z)
		-\om_2(X,\rho_1(Y)(Z))
		,\\
		\ass_{\circ}(Y,X,Z)&=(Y\circ X)\circ Z-Y\circ(X\circ Z)\\
		&=\rho_1(Y)(X)\circ_2 Z+\om_2(\rho_1(Y)(X),Z)-\rho_1(Y)(X\circ_2 Z)-Y\circ_1 \om_2(X,Z)-\om_1(Y,\om_2(X,Z)).	
	\end{align*}So
	\[ \begin{cases}
		Y\circ_1 \om_2(X,Z)=\om_2(\rho_1(Y)(X),Z)+\om_2(X,\rho_1(Y)(Z)),\\	
		\rho_1(Y)(X\circ_2 Z)=X\circ_2\rho_1(Y)(Z)
		+
		\rho_1(Y)(X)\circ_2  Z-\om_1(Y,\om_2(X,Z))-\rho_1(\rho_2(X)(Y))(Z).
	\end{cases} \]
\end{itemize}
Let us expand the second relation in \eqref{dAS}. We distinguish many cases.
\begin{itemize}
	\item 
	For $X,Y\in\h_1$, we get
	\begin{align*}
		A(X\circ Y)&=B_1(X\circ_1 Y)+\frac12X\circ_1Y+B_2(\om_1(X,Y))+\om_1(X,Y),\\
		AX\circ Y&=B_1(X)\circ_1Y+\om_1(B_1(X),Y)+\frac12X\circ_1Y+\frac12\om_1(X,Y),\\
		X\circ AY&=X\circ_1B_1(Y)+\om_1(X,B_1(Y))+\frac12 X\circ_1Y+\frac12\om_1(X,Y).	
	\end{align*}So
	\[ \begin{cases}
		B_1(X\circ_1 Y)=B_1(X)\circ_1Y+X\circ_1B_1(Y)+
		+\frac12 X\circ_1Y,\\
		B_2(\om_1(X,Y))=\om_1(B_1(X),Y)+\om_1(X,B_1(Y)).
	\end{cases} \]
	
	\item For $X,Y\in\h_2$, we get
	\begin{align*}
		A(X\circ Y)&=B_2(X\circ_2 Y)+X\circ_2Y+B_1(\om_2(X,Y))+\frac12\om_2(X,Y),\\
		AX\circ Y&=B_2(X)\circ_2Y+\om_2(B_2(X),Y)
		+X\circ_2Y+\om_2(X,Y),\\
		X\circ AY&=X\circ_2B_2(Y)+\om_2(X,B_2(Y))+ X\circ_2Y+\om_2(X,Y),\\
		-X\circ Y&=-X\circ_2 Y-\om_2(X,Y).	
	\end{align*}So
	\[ \begin{cases}
		B_2(X\circ_2 Y)=B_2(X)\circ_2Y+X\circ_2B_2(Y),\\
		B_1(\om_2(X,Y))=
		\om_2(B_2(X),Y)+\om_2(X,B_2(Y))+\frac12\om_2(X,Y).
	\end{cases} \]
	\item For $X\in\h_1$ and $Y\in\h_2$, we get
	\begin{align*}
		A(X\circ Y)&=B_2(\rho_1(X)(Y))+\rho_1(X)(Y),\\
		AX\circ Y&=\rho_1(B_1(X))(Y)+\frac12\rho_1(X)(Y),\\
		X\circ AY&=\rho_1(X)(B_2(Y))+\rho_1(X)(Y).
	\end{align*}So
	$ [B_2,\rho_1(X)]=\rho_1(B_1(X))+\frac12\rho_1(X).$
	\item For $X\in\h_2$ and $Y\in\h_1$, we get
	\begin{align*}
		A(X\circ Y)&=B_1(\rho_2(X)(Y))+\frac12\rho_2(X)(Y),\\
		AX\circ Y&=\rho_2(B_2(X))(Y)+\rho_2(X)(Y),\\
		X\circ AY&=\rho_2(X)(B_1(Y))+\frac12\rho_2(X)(Y),\\
		-X\circ Y&=-\rho_2(X)(Y).
	\end{align*}So
	\[ [B_1,\rho_2(X)]=\rho_2(B_2(X)). \esp
	B_2(\om_1(X,Y))=\om_1(B_1(X),Y)+\om_1(X,B_1(Y)). \]
\end{itemize}
The algebra $(\h_1,\circ_1,\prs)$ satisfies the hypothesis of Lemma \ref{useful1} and hence $\circ_1$ is trivial. 

To sum up the algebras $(\h_i,\circ_i,\prs_i)$ and the associated data $(\rho_1,\rho_2,\om_1,\om_1,B_1,B_2)$ satisfy the following properties: $\circ_1=0$, $B_1$ is skew-symmetric, $\tr(\rho_1(X))=0$ for any $X\in\h_1$ and
\[ \begin{cases}
	\langle\om_1(X,Y),Z\rangle_2=\langle \rho_2(Z)(X),Y\rangle_1+\langle \rho_2(Z)(Y),X\rangle_1,\quad X,Y\in\h_1,Z\in\h_2, \\
	\langle\om_2(X,Y),Z\rangle_1=\langle \rho_1(Z)(X),Y\rangle_2+\langle \rho_1(Z)(Y),X\rangle_2,\quad X,Y\in\h_2,Z\in\h_1.
\end{cases} \]

\begin{align*} \begin{cases}
		\langle X\circ_2 Y-Y\circ_2 X,Z\rangle_2=
		\langle Y\circ_2 Z,X\rangle_2-\langle  X\circ_2 Z,Y\rangle_2,\\
		\ass_{\circ_2}(X,Y,Z)-\ass_{\circ_2}(Y,X,Z)=
		\left(\langle Y,Z\rangle_2 X-\langle X,Z\rangle_2 Y \right),\\
		B_2(X\circ_2 Y)=B_2(X)\circ_2Y+X\circ_2B_2(Y),\\
		\langle B_2X,Y\rangle_2=-\langle B_2Y,X\rangle_2,\tr(\Li_X^{\circ_2})=-\tr(\rho_2(X)).
\end{cases} \end{align*}
\[ \begin{cases}
	B_2(\om_1(X,Y))=\om_1(B_1(X),Y)+\om_1(X,B_1(Y)),\\
	B_1(\om_2(X,Y))=
	\om_2(B_2(X),Y)+\om_2(X,B_2(Y))+\frac12\om_2(X,Y).	
\end{cases} \]

\[ \begin{cases}
	\rho_1(X)(Y\circ_2 Z)=Y\circ_2\rho_1(X)(Z)
	+\rho_1(X)(Y)\circ_2  Z-\rho_1(\rho_2(Y)(X))(Z)-\om_1(X,\om_2(Y,Z)),\\
	\rho_2(\rho_1(Y)(X))(Z)+\om_2(X,\om_1(Y,Z))=0,\\
	\rho_1(X)(\om_1(Y,Z))-\rho_1(Y)(\om_1(X,Z)) =0,\\
	\rho_2(X)(\om_2(Y,Z))-\rho_2(Y)(\om_2(X,Z)) +\om_2(X,Y\circ_2 Z)-\om_2(Y,X\circ_2 Z)-{\om_2([X,Y], Z)}=0,\\	
	\om_2(\rho_1(X)(Y),Z)+\om_2(Y,\rho_1(X)(Z))=0,\\
	X\circ_2(\om_1(Y,Z))=\om_1(\rho_2(X)(Y),Z)+
	\om_1(Y,\rho_2(X)(Z)){-}\langle Y,Z\rangle_1 X,\\
	[B_2,\rho_1(X)]=\rho_1(B_1(X))+\frac12\rho_1(X),\\
	[B_1,\rho_2(X)]=\rho_2(B_2(X)).	\end{cases} \]
Let us show that relations $[B_2,\rho_1(X)]=\rho_1(B_1(X))+\frac12\rho_1(X)$,
$[B_1,\rho_2(X)]=\rho_2(B_2(X))$ and the definition of $\om_i$ implies the system
\[ \begin{cases}
	B_2(\om_1(X,Y))=\om_1(B_1(X),Y)+\om_1(X,B_1(Y)),\\
	B_1(\om_2(X,Y))=
	\om_2(B_2(X),Y)+\om_2(X,B_2(Y))+\frac12\om_2(X,Y)	
\end{cases} \]
and hence this system is redundant. Indeed,
\begin{align*}
	\langle B_2(\om_1(X,Y)),Z\rangle_2&=-\langle \rho_2(B_2(Z))(X),Y\rangle_1-\rho_2(B_2(Z))(Y),X\rangle_1	\\
	&=-\langle B_1(\rho_2(Z)(X)),Y\rangle_1+\langle \rho_2(Z)(B_1(X)),Y\rangle_1
	-\langle B_1(\rho_2(Z)(Y)),X\rangle_1+\langle \rho_2(Z)(B_1(Y)),X\rangle_1,\\
	\langle \om_1(B_1(X),Y),Z\rangle_2&=\langle \rho_2(Z)(B_1(X)),Y\rangle_1+\langle \rho_2(Z)(Y),B_1(X)\rangle_1,\\
	\langle \om_1(X,B_1(Y)),Z\rangle_2&=\langle \rho_2(Z)(B_1(Y)),X\rangle_1+\langle \rho_2(Z)(X),B_1(Y)\rangle_1
\end{align*}and the first relation follows. The second relation follows in a similar way.

\end{document}